\documentclass[11pt]{amsart}
\usepackage{fullpage}
\usepackage{amsfonts}
\usepackage{amsmath}
\usepackage{amssymb}
\usepackage{amsthm}
\newtheorem{theorem}{Theorem}

\newtheorem{cor}[theorem]{Corollary}

\newtheorem{lemma}[theorem]{Lemma}
\newtheorem{prop}[theorem]{Proposition}
\newtheorem{defin}[theorem]{Definition}

\newcommand{\sph}{\mathbb{S}^{n-1}}

\newcommand{\rn}{\mathbb{R}^{n}}
\newcommand{\p}{Pos_{n,2d}}

\newcommand{\sq}{Sq_{n,2d}}
\newcommand{\ip}[2]{\langle #1,#2 \rangle }
\newcommand{\st}{\hspace{2mm} \mid \hspace{2mm}}
\newcommand{\binomi}[2]{\left(\genfrac{}{}{0 pt}{}{#1}{#2}\right)}
\numberwithin{theorem}{section}
\numberwithin{equation}{section}
\begin{document}
\author{Grigoriy Blekherman}

\title{Dimensional Differences between Nonnegative Polynomials and Sums of Squares}
\begin{abstract}
We study dimensions of the faces of the cone of nonnegative polynomials and the cone of sums of squares; we show that there are dimensional differences between corresponding faces of these cones. These dimensional gaps occur in all cases where there exist nonnegative polynomials that are not sums of squares. As either the degree or the number of variables grows the gaps become very large, asymptotically the gaps approach the full dimension of the vector space of polynomials in $n$ variables of degree $2d$. The gaps occur generically, they are not a product of selecting special faces of the cones. Using these dimensional differences we show how to derive inequalities that separate nonnegative polynomials from sums of squares; the inequalities will hold for all sums of squares, but will fail for some nonnegative polynomials.
\end{abstract}
\maketitle

\section{Introduction}
The relationship between nonnegative polynomials and sums of squares has been studied since Hilbert's seminal paper of 1888. In it Hilbert showed that a nonnegative polynomial in $n$ variables of even degree $2d$ has to be a sum of squares of polynomials only in the following cases: the polynomial is univariate $n=1$, the polynomial is quadratic $2d=2$, or the polynomial is in 2 variables and has degree 4, $n=2$ and $2d=4$. In all other cases he proved existence of nonnegative polynomials that are not sums of squares. It is remarkable that Hilbert's proof was existential and the first explicit nonnegative polynomial that is not a sum of squares was found only 70 years later by Motzkin \cite{Rez1}, \cite{Rez2}.

Hilbert then showed that every nonnegative polynomial in 2 variables is a sum of squares of rational functions and Hilbert's 17th problem asked to show that this is true for any number of variables. This was shown in the 1920's by Artin and Schreier (see \cite{Prest}). However, there is no known algorithm to compute this representation and we may be forced to use denominators and numerators of very large degree compared to the original degree $2d$. While there are bounds on the degree of polynomials used to make rational functions, they are not encouraging from the point of view of efficiently computing such representations (see \cite{Prest}). 

A nonnegative polynomial can be homogenized and it will remain nonnegative. Therefore for the remainder of this paper we will restrict our attention to the case of homogeneous polynomials (forms).

Let $P_{n,2d}$ be the vector space of forms in $n$ variables of degree $2d$. For a fixed number of variables $n$ and degree $2d$ nonnegative polynomials and sums of squares form closed convex cones in $P_{n,2d}$. We call these cones $Pos_{n,2d}$ and $Sq_{n,2d}$ respectively:

$$Pos_{n,2d}=\left\{p \in P_{n,2d}\hspace{2mm} | \hspace{2mm} p(x)\geq 0 \hspace{2mm} \text{for all} \hspace{2mm} x \in \rn\right\},$$

$$Sq_{n,2d}=\left\{p \in P_{n,2d} \hspace{2mm} | \hspace{2mm} p(x)=\sum q_i^2  \hspace{2mm} \text{for some} \hspace{2mm} q_i \in P_{n,d} \right\}.$$

These cones are very interesting convex objects but their structure and precise relationship with each other is not very well understood except for the cases of $n=2$, the univariate case for nonhomogeneous polynomials, and the case $2d=2$ (see \cite{Sasha1} Sections II.11 and II.12 for these cases). We studied some convexity properties such as the coefficient of symmetry and maximal volume ellipsoids of these cones in \cite{Me1}.

In \cite{Me2} we have shown that if the degree $2d$ is fixed and the number of variables $n$ grows then asymptotically there are significantly more nonnegative polynomials than sums of squares if the degree $2d$ is at least 4. However the relationship between the cones for small values of $n$ and $d$ is not clear.

The precise relationship between the cones $Pos_{n,2d}$ and $Sq_{n,2d}$ is interesting because of issues of computational complexity and practical testing for nonnegativity. It is known that testing whether a polynomial is nonnegative is NP-hard already when the degree is 4 \cite{NP}. On the other hand testing whether a polynomial is a sum of squares can be reduced to a semidefinite programming problem and it is practically quite fast \cite{pablo}.

In this paper we study the faces of the cone of nonnegative polynomials and the cone of sums of squares. In particular we are interested in the dimensions of the faces. It seems that outside of the simple cases of $n=2$ and $2d=2$ very little is known about the possible dimensions of the faces of these cones. 
\subsection{Faces of $\p$ and $\sq$.}

It is easy to describe the faces of $Pos_{n,2d}$. The boundary of the cone $Pos_{n,2d}$ consists of all the forms with at least one zero while the interior consists of strictly positive forms. A facet of $\p$ consists of all the forms with a prescribed zero. If we let $S$ be a set of points in $\rn$ (we should think of $S$ projectively as points in $\mathbb{RP}^{n-1}$) then the forms vanishing on all points of $S$ form a face of $\p$ which we call $Pos_{n,2d}(S)$:
$$\p(S)=\{p \in \p \st p(s)=0 \hspace{2mm} \text{for all} \hspace{2mm} s \in S\}.$$

Moreover any face of $\p$ has a description of this form and the set $S$ can be chosen to be finite. We note that despite this simple description the facial structure of $\p$ "should" be very difficult to fully describe because the problem of testing for nonnegativity is known to be NP-hard.

The faces of the cone of sums of squares are much harder to describe. We will only look at the faces of $\sq$ that have description analogous to the faces of $\p$. For a (projective) set of points $S$ in $\rn$ we let $\sq(S)$ be the face of $\sq$  of forms that vanish on all points of $S$:

$$\sq(S)=\{p \in \sq \st p(s)=0 \hspace{2mm} \text{for all} \hspace{2mm} s \in S\}.$$

We will study the dimensions of the faces $\p(S)$ and $\sq(S)$ and we will establish large dimensional gaps between the faces for the same set $S$. It is possible to view Hilbert's original proof of existence of nonnegative polynomials that are not sums of squares as establishing a dimensional gap of this type. This dimensional point of view was first made explicit in \cite{Rez2}.

\subsection{Bounds on the Dimensions of Faces.}

For a set of points $S \in \rn$ let $I_{1,d}(S)$ be the vector subspace of $P_{n,d}$ of forms that vanish to order at least 1 on $S$:
$$I_{1,d}(S)=\{p \in P_{n,d} \st p(s)=0 \hspace{2mm} \text{for all} \hspace{2mm} s \in S\}.$$

If a nonnegative form $p$ vanishes at a point $s$ then $p$ attains its minimum at $s$ and therefore $p$ must be singular at $s$ or in other words $p$ must vanish to order 2 on $s$. We let $I_{2,2d}(S)$ be the vector subspace of $P_{n,2d}$ of forms that vanish to order at least 2 on every point of $S$:
$$I_{2,2d}(S)=\left\{p \in P_{n,2d} \st \frac{\partial p}{\partial x_i}(s)=0 \hspace{2mm} \text{for all} \hspace{2mm} s\in S \hspace{2mm} \text{and all} \hspace{2mm} i=1\ldots n \right\}.$$

Since every nonnegative form that is zero on $s$ must vanish to order 2 on $s$ it follows that the face $\p(S)$ is contained in $I_{2,2d}(S)$. Now $\p(S)$ is a face of $\p$ and therefore it is a convex cone. We are interested in finding the dimension of $\p(S)$ and in particular we will be interested in showing that in many cases $\p(S)$ is full-dimensional in $I_{2,2d}(S)$. We will also provide some examples where the full-dimensionality does not hold in Sections \ref{special} and \ref{fewer}.

The dimension of $I_{2,2d}(S)$ has been extensively studied previously. Vanishing to order 2 at a point imposes $n$ conditions on the forms in $P_{n,2d}$. Therefore one would expect that generically $I_{2,2d}(S)$ has codimension $n|S|$ in $P_{n,2d}$. Indeed the Alexander-Hirschowitz Theorem states that generically this is the case, except for a small number of exceptions \cite{Miranda}. By showing full-dimensionality of $\p(S)$ in $I_{2,2d}(S)$ we get a handle on the dimension of the face $\p(S)$.

We will call a finite set $S \in \rn$ $d$-independent if $S$ satisfies the following two conditions:
\begin{align}
&\label{first cond}\text{The forms in $I_{1,d}(S)$ share no common zeroes outside of $S$. In other words the}\\
&\nonumber \text{conditions of vanishing on $S$ force no additional zeroes on forms of degree $d$.}\\
&\label{second cond}\text{For any $s \in S$ the forms that vanish to order 2 on $s$ and vanish on the rest of $S$}\\
&\nonumber \text{form a vector space of codimension $|S|+n-1$ in $P_{n,d}$.}&
\end{align}

\noindent The second condition simply states that the constraints of vanishing on $S$ and additionally double vanishing at any point $s \in S$ are all linearly independent.

In Section \ref{dimsection} we will show that if a set $S$ is $d$-independent then the face $\p(S)$ is full dimensional in $I_{2,2d}(S)$. We will also show that the set of configurations that are $d$-independent is open and therefore in order to show that $d$-independence is a generic condition for sets of fixed size $k$ we simply need to provide a single example.

In Section \ref{dset} we construct a $d$-independent set of size $\binom{n+d-1}{d}-n$. This shows that $d$-independence is a generic condition for finite sets of size at most $\binom{n+d-1}{d}-n$ in $\rn$.

In Section \ref{gapsection} we will establish large gaps between faces $\p(S)$ and $\sq(S)$. For sums of squares we will use a different dimensional approach. Let $I^{[2]}_{2d}(S)$ be the vector subspace of $P_{n,2d}$ that is spanned by squares of forms from $I_{1,d}(S)$:
$$I_{2d}^{[2]}(S)=\left\{p \in P_{n,2d} \hspace{2mm} \mid \hspace{2mm} p=\sum \alpha_iq_i^2 \hspace{2mm} \text{for some} \hspace{2mm} q_i \in I_{1,d}(S) \hspace{2mm} \text{and}\hspace{2mm} \alpha_i \in \mathbb{R} \right\}.$$

If a form $p$ is a sum of squares, $p=\sum q_i^2$ and $p$ vanishes on $S$ then each $q_i$ must vanish on $S$ and therefore $p$ is a sum of squares of forms from $I_{1,d}(S)$. Thus it follows that $\sq(S)$ is contained in $I_{2d}^{[2]}(S)$. It is easy to see that $\sq(S)$ is actually full dimensional in $I_{2d}^{[2]}(S)$ since we can pick a basis of $I_{2d}^{[2]}(S)$ consisting of squares and nonnegative linear combinations of these squares will lie in $\sq(S)$.

Therefore we can restrict our attention to the dimension of $I_{2d}^{[2]}(S)$. To bound this dimension we observe that the squares from $I_{1,d}(S)$ can span a vector space of dimension at most $\binom{\dim I_{1,d}(S)+1}{2}$. When the set $S$ is fairly small this bound is excessive and it would be interesting to improve on it. However when the set $S$ is large the bound is often optimal.

Using the above bounds we will show that there exist dimensional differences between faces $\p(S)$ and $\sq(S)$ of size
$$Gap_{n,2d}=\binom{n+2d-1}{2d}-n\binom{n+d-1}{d}+\binom{n}{2}.$$
We note that the numbers $Gap_{n,2d}$ are zero in all the cases where the cones $\p$ and $\sq$ are equal. However $Gap_{n.2d}$ are strictly positive in the cases where there exist nonnegative forms that are not sums of squares. In the smallest cases $n=4$, $2d=4$ and $n=3$, $2d=6$ where $\p$ is strictly bigger than $\sq$ the gap number is 1.

However, as either $n$ or $d$ grow we can see that the dimensional gap $Gap_{n,2d}$ between faces of $\p$ and $\sq$ grows and asymptotically it approaches the full dimension of the vector space $P_{n,2d}$.

In Section \ref{sixpoints} we will completely describe the situation when $S$ is a set of 6 points in $\mathbb{R}^4$. There will be generically a gap of 1 dimension between $\p(S)$ and $\sq(S)$. We will explicitly describe an extra linear constraint satisfied by sums of squares, that is not satisfied by nonnegative forms and we will also provide explicit examples of forms that double vanish on $S$, but are not spanned by squares. These forms can be used together with Lemma \ref{ext} to construct explicit nonnegative forms that are not sums of squares.

In Section \ref{sectiongenineq} we will use these dimensional differences to derive quadratic inequalities that separate nonnegative forms from sums of squares. These inequalities will hold on the cone of sums of squares $\sq$ but they will fail for some nonnegative forms. In Section \ref{sectionexplicitineq} we will give an explicit example of such inequality for forms of degree 4 in 4 variables.

We begin by giving some explicit examples of dimensional gaps between nonnegative forms and sums of squares and also examples of dimensional differences between nonnegative and double vanishing forms. Some details of dimension counts will be omitted, but the proofs will be given in more generality in later sections.

\section{Examples}

\subsection{First Gap: Six Points in $\mathbb{R}^4$.}\label{sectionfirstgap}\hspace{1mm}\\
\indent Let $S$ be the set of 6 points in $\mathbb{R}^4$ that have 1 in two coordinates and 0 in the two other coordinates. There are $\binom{4}{2}=6$ such points. We label the points $s_{ij}$ by the two coordinates in which 1 appears.

Let $I_{1,2}(S)$ be the vector space of all forms of degree 2 that vanish on $S$ and let $I_{2,4}(S)$ be the vector space of forms of degree 4 that double vanish on $S$. Also let $I_4^{[2]}(S)$ be the vector space of forms of degree 4 spanned by squares from $I_{1,2}(S)$. This example will be generalized in Section \ref{dset} where our dimension counts will be rigorously proved in more generality.

It is not hard to show that the set $S$ is $2$-independent. In particular, requiring 6 zeroes on the points of $S$ imposes 6 linearly independent conditions on quadratic forms and therefore the dimension of $I_{1,2}(S)$ is $\binom{5}{2}-6=4$. There is a particularly nice basis for $I_{1,2}(S)$ in which every form factors:
\begin{eqnarray*}
Q_1=x_1(x_1-x_2-x_3-x_4), \hspace{4mm}& Q_2=x_2(x_2-x_1-x_3-x_4)\\
Q_3=x_3(x_3-x_1-x_2-x_4), \hspace{4mm}& Q_4=x_4(x_4-x_1-x_2-x_3).
\end{eqnarray*}

Using this basis for $I_{1,2}(S)$ it is easy to check that the dimension of $I_4^{[2]}$ is $\binom{5}{2}=10$, which happens because all pairwise products of $Q_i$'s are linearly independent in $P_{4,4}$. On the other hand the Alexander-Hirschowitz theorem leads us to expect that the dimension of vector space $I_{2,4}(S)$ of double vanishing forms on $S$ is $\binom{7}{4}-4 \cdot 6=11$. It is not hard to verify that this dimension count is correct for our set $S$.

By Corollary \ref{dindep} $2$-independence of $S$ implies that the face $Pos_{4,4}(S)$ is full dimensional in the vector space $I_{2,4}(S)$. Therefore we get a face $Pos_{4,4}(S)$ which has dimension 11, while the corresponding face $Sq_{4,4}(S)$ has dimension 10.

We briefly explain why the face $Pos_{4,4}(S)$ is full dimensional in $I_{2,4}(S)$. Consider the polynomial $$Q=Q_1^2+Q_2^2+Q_3^2+Q_4^2.$$

It is not hard to show that $Q$ has no zeroes outside of $S$ and furthermore at every $s_{ij}\in S$ the Hessian of $Q$ is positive definite on the vector subspace $s_{ij}^{\perp}$ consisting of vectors perpendicular to $s_{ij}$. This suffices to show that $Q$ can be perturbed in any direction by a double vanishing polynomial in $I_{2,4}(S)$ and it will still remain nonnegative (See Lemma \ref{ext}). It follows that the face $Pos_{4,4}(S)$ of the cone of nonnegative polynomials vanishing on $S$ is full dimensional in $I_{2,4}(S)$ and thus has dimension 11, while the face of the cone of sums of squares $Sq_{4,4}(S)$ has dimension 10.

Since there is a gap of one dimension between the faces there exist double vanishing forms in $I_{2,4}(S)$ that are not spanned by sums of squares. Also, it follows that sums of squares must satisfy an extra linear condition that is not satisfied by the double vanishing forms. It is not hard to check that $x_1x_2x_3x_4$ is indeed a doubly vanishing form on $S$, and using our special basis it is easy to show that $x_1x_2x_3x_4$ is not in $I_4^{[2]}(S)$.

The extra linear constraint satisfied by any form $p$ in $I_4^{[2]}(S)$ can be given by:
$$16p(1,0,0,0)=p\left(1,-1,-1,-1\right).$$
On the other hand this constraint is not satisfied by $x_1x_2x_3x_4$.

We show in Section \ref{sixpoints} that the situation in the case of 6 general points in $\mathbb{R}^4$ is very similar. There is a gap of 1 dimension and we also provide explicitly the extra linear condition satisfied by sums of squares and a double vanishing form that is not in the span of sums of squares.

\subsection{A Special Configuration.}\label{special}\hspace{1mm}\\
\indent We now add the point $(1,1,1,1)$ to the six point set $S$ from above to form a 7 point set $S'$ in $\mathbb{R}^4$. We claim that the nonnegative polynomials vanishing on $S'$ do not form a full dimensional convex set in the vector space $I_{2,4}(S')$ of forms of degree 4 double vanishing on $S'$.

The dimension of $\dim I_{2,4}(S')$ has to be at least 7, since we impose at most $7\cdot 4=28$ constraints on a 35 dimensional vector space $P_{4,4}$. It is not hard to show that $\dim I_{2,4}(S')=7$, which is the expected dimension by the Alexander-Hirchowitz Theorem. On the other hand we will see that $\dim Pos_{4,4}(S')=\dim Sq_{4,4}(S')=6$ so both nonnegative polynomials and sums of squares form convex cones of dimension 6 in the 7 dimensional vector space of double vanishing forms. This is a special situation among configurations of 7 points in $\mathbb{R}^4$.

Using the forms $Q_i$ defined above as the basis of $S$ it is clear that $R_1=Q_1-Q_4$, $R_2=Q_2-Q_4$ and $R_3=Q_3-Q_4$ form a basis of $I_{1,2}(S')$, since $Q_i(1,1,1,1)=-2$ for all $i$. After simplification we see that
\begin{align*}
R_1=(x_1-x_4)(x_1&+x_4-x_2-x_3), \hspace{3mm} R_2=(x_2-x_4)(x_2+x_4-x_1-x_3), \\
&R_3=(x_3-x_4)(x_3+x_4-x_2-x_3).
\end{align*}

The set $S'$ is not 2-independent. It is easy to show that the forms $R_i$ have no common zero outside of $S'$. Therefore $S'$ forces no additional zeroes for forms of degree 2 and the first condition \eqref{first cond} for 2-independence is satisfied.

Now lets look at the second condition \eqref{second cond}. Since $|S'|=7$ we would need for any $s \in S'$ the vector space of forms vanishing on $S'$ and double vanishing on $s$ to have codimension 10 in $P_{4,2}$. The dimension of $P_{4,2}$ is 10 and therefore for any $s \in S'$ we would need to have no nonzero forms that are singular on $s$ and vanishing on the rest of $S'$. However the form
$$(x_1-x_2)^2-(x_3-x_4)^2$$
is actually singular at $(1,1,0,0)$ and it vanishes on the rest of $S'$.

It is easy to show that all pairwise products of forms $R_i$ are linearly independent and therefore the dimension of the vector space $I_4^{[2]}(S')$ spanned by squares from $I_{1,2}(S)$  is 6. This implies that the dimension of the face $Sq_{4,4}(S')$ is also 6. We now show that the dimension of the face $Pos_{4,4}(S')$ is 6 as well.

We observe that the set $S'$ is fixed by the action of symmetric group $\mathbb{S}_4$ which acts by permuting the coordinates. Therefore there is a natural action of $\mathbb{S}_4$ on the double vanishing vector space $I_{2,4}(S')$. It is not hard to see that there is a 2 dimensional subspace of $I_{2,4}(S')$ spanned by symmetric forms, i.e. the forms fixed by $\mathbb{S}_4$. A basis of this subspace is given by the forms $F_1$ and $F_2$:

\begin{equation*}
F_1 =\sum_{i <j} (Q_i-Q_j)^2,
\end{equation*}

\begin{equation*}
F_2 =(Q_1+Q_2+Q_3+Q_4)^2-64x_1x_2x_3x_4.
\end{equation*}

The rest of $I_{2,4}(S')$ is made of two irreducible representations of $\mathbb{S}_4$ corresponding to partitions $(2,2)$ and $(3,1)$. This happens since $I_{1,2}(S')$ corresponds to the standard representation of $\mathbb{S}_4$; this is easy to see from the basis of $R_i$. Now $I^{[2]}_4$ is the symmetric square of $I_{1,2}(S')$ and it is known to split into 3 irreducibles: the trivial representation, spanned by $F_1$, and representation corresponding to partitions $(2,2)$ and $(3,1)$. Since $I^{[2]}_4(S')$ has dimension 6 and $F_2$ is not in $I^{[2]}_4(S')$ it follows that $F_2$ and $I^{[2]}_4(S')$ together span $I_{2,4}(S')$.

Thus $I_{2,4}(S')$ splits into two copies of the trivial representation (spanned by $F_1$ and $F_2$) and two more representations of $\mathbb{S}_4$ corresponding to partitions $(2,2)$ and $(3,1)$. For more details on representations of the symmetric group we refer to \cite{Fulton}.

Now lets pick a basis of $I_{2,4}(S')$ consisting of $F_1$, $F_2$ and 5 forms $G_i$ forming a basis of the other two irreducible sub-representations. Let's take a form $p$ in $Pos_{4,4}(S')$. We can write $p=\alpha_1F_1+\alpha_2F_2+\sum \beta_iG_i$. We claim that the coefficient $\alpha_2$ of $F_2$ must be zero, and therefore $Pos_{4,4}(S')$ is not full dimensional in $I_{2,4}(S')$.

Suppose not. Symmetrize $p$ with respect to the action of $\mathbb{S}_4$ to obtain a new polynomial $\bar{p}$:

$$\bar{p}=\frac{1}{24}\sum_{g \in \mathbb{S}_4} p(gx).$$

Since $p$ is nonnegative, it follows that $\bar{p}$ is an average of nonnegative forms and therefore also nonnegative. By elementary representation theory, only the trivial representation components of $p$ survive the averaging process, so

$$\bar{p}=\alpha_1F_1+\alpha_2F_2.$$

Since $F_1$ is a sum of squares it is clearly nonnegative and therefore $F_1$ is in $Pos_{4,4}(S')$. It is easy to show that for $\bar{p}=\alpha_1F_1+\alpha_2F_2$ to be nonnegative the coefficient $\alpha_1$ of $F_1$ must be positive. But then, since both $F_1$ and $\alpha_1F_1+\alpha_2F_2$ are in $Pos_{4,4}(S')$, we see that $F_1 +\epsilon F_2$ must be in $Pos_{4,4}(S')$ for any small enough $\epsilon$ because $Pos_{4,4}(S')$ is a convex cone.

Now lets restrict $F_1$ and $F_2$ to the line $(x,1,1,1)$. We can see that $F_1$ on this line is equal to $3(x-1)^4$ while $F_2$ on this line is equal to $(x-1)^3(x-9)$. Since $F_1$ has a zero of degree 4 at $(1,1,1,1)$ along this line and $F_2$ has zero of degree 3 it follows that $F_1+\epsilon F_2$ cannot be nonnegative for sufficiently small $\epsilon$.

\subsection{No Nonnegative Polynomials among Double Vanishing Forms.} \label{fewer}\hspace{1mm}\\
\indent As we have seen above it does not have to happen that the face $\p(S)$ is full dimensional in the vector space of double vanishing forms $I_{2,2d}(S)$. We now provide an example where there are no nonnegative forms with a specified zero set, while dimension count shows that many double vanishing forms with this zero set exist.

In order to exclude nonnegative forms we need to require at least $\binom{n+d-1}{d}$ zeroes, otherwise there will exist squares with the specified zeroes. We restrict ourselves to the case of $2d=4$.

We show below that if take a set $S$ of cardinality $\binom{n+1}{2}$ then generically there will exist no nonnegative polynomials that vanish on $S$ (by generic over the real numbers we mean that it holds on an open set of configurations). This suggests an explanation for one aspect of why it is hard to construct nonnegative polynomials that are not sums of squares: if we require enough zeroes to exclude all squares then generically we will not have any nonnegative polynomials left either.

Let $S_{n,2}$ be the set of all vectors in $\rn$ that are partitions of $2$: $$S_{n,2}=\left\{s=(s_1,\dots,s_n)\in \rn \st s_i \geq 0, s_i \in \mathbb{Z},\hspace{1mm} \text{and}\hspace{2mm} \sum s_i=2\right\}.$$

$S_{n,2}$ consists of all vectors $s_{ij}$ with 1 in coordinates $i$ and $j$ and 0 in all other coordinates together with vectors $2e_i$, where $e_i$ are the standard basis vectors.

We claim that the vector space $I_{2,4}(S_{n,2})$ of double vanishing forms on $S_{n,2}$ of degree 4 is spanned by forms $x_ix_jx_kx_l$, where $i,j,k,l$ are distinct indices. In particular we need $n \geq 4$ for the vector space to be non-empty.

Let $p$ be a nonzero form in $I_{2,4}(S_{n,2})$. Since $p$ double vanishes at each standard basis vector $e_i$ it follows that $p$ cannot contain monomials $x_i^4$ or $x_i^3x_j$ for any indices $i,j$. Of the remaining monomials of degree 4 only $x_i^2x_j^2$ does not vanish on the point $s_{ij}$. Since $p$ does vanish on $s_{ij}$ it follows that $p$ does not contain any of these monomials either.

The only monomials left that we are allowed to use are of the form $x_i^2x_jx_k$ and $x_ix_jx_kx_l$. We can exclude $x_i^2x_jx_k$ as follows: let $\alpha$ be the coefficient of $x_i^2x_jx_k$, let $\beta$ be the coefficient of $x_ix_j^2x_k$ and let $\gamma$ be the coefficient of $x_ix_jx_k^2$. The monomials $x_i^2x_jx_k$ and $x_ix_j^2x_k$ are the only monomials left whose $k$-th partial derivative does not vanish on $s_{ij}$. Therefore we see that $\alpha=-\beta$, since $p$ double vanishes on $s_{ij}$. Similarly, $\alpha=-\gamma$ and $\beta=-\gamma$. This can only happen if $\alpha=\beta=\gamma=0$. Thus it follows that $p$ can only contain monomials of the form $x_ix_jx_kx_l$. It is easy to see that each of the monomials $x_ix_jx_kx_l$ does indeed double vanish on $S_{n,2}$, which proves our claim.

It is clear that monomials $x_ix_jx_kx_l$ do not span any nonnegative forms so $I_{2,4}(S_{n,2})$ does not contain any nonnegative forms and it is nonempty for $n \geq 4$. To prove that this situation is generic we observe that for a set $S$ the condition that $I_{2,4}(S)$ does not contain any nonnegative forms is clearly "open", i.e. we can perturb the points in $S$ by a small enough $\epsilon$ to obtain a new set $S'$ and $I_{2,4}(S')$ will not contain any nonnegative polynomials either.

We need to be a little bit careful however. The points $S_{n,2}$ are in fact in very special position from the point of view of Alexander-Hirschowitz theorem. The vector space $I_{2,4}(S_{n,2})$ has a larger dimension than is expected generically. Therefore if we perturb the points to a generic configuration we may end up with an empty vector space of double vanishing forms. Indeed, it is easy to check that we need $n \geq 7$ before the generic dimension $\binom{n+3}{4}-n\binom{n+1}{2}$ becomes positive. However for large $n$ this dimension will asymptotically approach the dimension of the whole space of forms of degree 4, which is $\binom{n+3}{4}$.

\section{Dimension of faces of $\p$}\label{dimsection}

Let $S$ be a finite set in $\rn$. We will find the dimension of $\p(S)$ by establishing that in many cases $\p(S)$ is actually full dimensional in the vector space $I_{2,2d}(S)$ of double vanishing forms on $S$ of degree $2d$. The dimension of $I_{2,2d}(S)$ has been well studied and in the case of a generic set $S$ the dimension of $I_{2,2d}(S)$ is provided by Alexander-Hirschowitz theorem \cite{Miranda}.

Generically one expects that every double zero contributes $n$ new linear conditions and Alexander-Hirschowitz theorem states that this is indeed the case, with a small list of exceptions. Therefore generically be expect that:

$$\dim I_{2,2}d(S)=\dim P_{n,2d}-n|S|.$$

\noindent However, for any set $S$ we know that

$$\dim I_{2,2d}(S)\geq \dim P_{n,2d}-n|S|,$$

\noindent since we impose at most $n|S|$ linearly independent conditions.

We will establish full dimensionality of $\p(S)$ by finding a form $p \in \p(S)$ to which we can add a suitably small multiple of any double vanishing form and it will remain nonnegative:

$$p+\epsilon q \in \p(S) \hspace{2mm} \text{for some sufficiently small} \hspace{2mm} \epsilon \hspace{2mm} \text{and any} \hspace{2mm} q \in I_{2,2d}(S).$$

The form $p$ can be viewed as a certificate of full dimensionality of $\p(S)$. The important point is that $p$ can be any form, in particular we will focus on finding such $p$ that is a sum of squares. This approach follows that of \cite{Rez2} and indeed it can be traced to the original proof of Hilbert.

For a form $p$ let the Hessian $H_p$ of $p$ be the matrix of second derivatives of $p$:

$$H_p=(h_{ij}), \hspace{5mm} \text{where} \hspace{5mm} h_{ij}=\frac{\partial^2 p}{\partial x_i \partial x_j}.$$

We will need an "extension lemma" which follows from Lemma 3.1 of \cite{Rez2}. We note that if a form $p$ vanishes at a point $s$ then by homogeneity $p$ vanishes on a line through $s$. Therefore the vector $s$ is in the kernel of the Hessian of $p$ at $s$: $H_p(s)s=0$.

If a form $p$ is nonnegative then its Hessian at any zero $s$ is positive semidefinite, since 0 is a minimum for $p$. We call a nonnegative form $p$ \textit{round} at a zero $s$ if the Hessian of $p$ at $s$ is positive definite on the subspace $s^{\perp}$ of vectors perpendicular to $s$:
$$p \hspace{2mm} \text{is round at a zero} \hspace{2mm} s \hspace{2mm} \text{if} \hspace{2mm} H_p(s) \hspace{2mm} \text{is positive definite on} \hspace{2mm} s^{\perp}. $$

\noindent For a form $p$ we will let $Z(p)$ denote the zero set of $p$ in $\mathbb{RP}^{n-1}$.

\begin{lemma}\label{ext}
Let $p$ be a nonnegative form with a finite zero set $Z(p)$, and let $q$ be a form such that $q$ vanishes to order 2 on every point in the zero set of $p$. Furthermore suppose that $p$ is round at every point in $Z(p)$.
Then for a sufficiently small $\epsilon$, the form $p+\epsilon q$ is nonnegative.
\end{lemma}

Now we have an immediate corollary:

\begin{cor} \label{ext2}
Let $S$ be a finite set in $\rn$. Suppose that we can find a nonnegative form $p$ in $\p(S)$ such that $S$ is the zero set of $p$ projectively: $Z(p)=S$ and $p$ is round at every point $s \in S$. Then the face $\p(S)$ is full dimensional is the vector space $I_{2,2d}(S)$ of double vanishing forms on $S$.
\end{cor}
\begin{proof}
Let $p \in \p(S)$ be as above. Then by Lemma \ref{ext} for any $q \in I_{2,2d}(S)$ we have $p+\epsilon q \in \p(S)$ for sufficiently small $\epsilon$. Since $\p(S)$ is a convex set it follows that is is full dimensional in $I_{2,2d}(S)$.
\end{proof}

\subsection{Sum of Squares Certificate.}Now suppose that for a finite set $S \in \rn$ we can find a sum of squares form $p=q_1^2+\ldots+q_n^2$ such that $Z(p)=S$ and $p$ is round at every $s \in S$. Since $p$ vanishes on $S$ it implies that all $q_i$ vanish on $S$ and the set $S$ has no forced zeroes for forms of degree $d$. Or in other words, $S$ is cut out by the intersection of hypersurfaces $q_i=0$.

The Hessian of $p$ is the sum of the Hessians of $q_i^2$:$$H_p=\sum_i H_{q^2_i}.$$ Since $q_k(s)=0$ for all $k$ and $s \in S$ it follows that

$$\frac{\partial^2 q_k^2}{\partial x_i \partial x_j}(s)=2\frac{\partial q_k}{\partial x_i}(s)\frac{\partial q_k}{\partial x_j}(s).$$

Therefore we see that the Hessian of $q^2_k$ at any $s \in S$ is actually double the tensor of the gradient of $q_k$ at $s$ with itself: $$H_{q^2_k}(s)=2\nabla q_k\otimes \nabla q_k(s).$$ Since the Hessian of $p$ at $s$ is positive definite on $s^{\perp}$, and it is the sum of the gradients of $q_k$ it follows that the gradients of $q_k$ at $s$ span $s^{\perp}$.

The gradients of $q_k$ at any $s \in S$ cannot span more than $s^{\perp}$ since $\ip{\nabla  q_k}{s}=d\cdot q_k(s)=0$. The gradients of $\nabla q_k$ span a vector space of dimension $n-1$ and thus the forms that double vanish at one $s \in S$ and vanish at all other points in $S$ form a vector space of codimension $n-1$ in $I_{1,d}(S)$.

These two conditions are sufficient for $\p(S)$ to be full dimensional in $I_{2,2d}(S)$.

\begin{theorem}\label{gapzz}
Let $S$ be a finite set in $\rn$ such that $S$ forces no additional zeroes for forms of degree $d$ and for any $s \in S$ the forms in $I_{1,d}(S)$ that double vanish at $s$ form a vector space of codimension $n-1$ in $I_{1,d}(S)$. Then $\p(S)$ is a full dimensional convex cone in $I_{2,2d}(S)$.
\end{theorem}

\begin{proof}
Let $q_1, \ldots, q_k$ be a basis of $I_{1,d}(S)$. We claim that $p=\sum q_i^2$ has the properties of Lemma \ref{ext} and therefore by Corollary \ref{ext2} convex cone $\p(S)$ is full dimensional in $I_{2,2d}(S)$.

Since $S$ has no forced zeroes and we picked a basis of $I_{1,d}(S)$ it follows that $q_k$ have no common zeroes outside of $S$ and thus $Z(p)=S$ projectively.

Now choose $s \in S$. Since the forms in $I_{1,d}(S)$ that double vanish at $s$ form a vector space of codimension $n-1$ in $I_{1,d}(S)$ it follows that the gradients of $q_k$ at $s$ must span a vector space of dimension $n-1$. Since $\ip{\nabla q_k}{s}=0$ for all $k$ it follows that they actually span $s^{\perp}$. By the argument about Hessians above it follows that the Hessian of $p$ is positive definite on $s^{\perp}$.
\end{proof}

To avoid working with degenerate configuration instead of requiring that for any $s \in S$ the forms in $I_{1,d}(S)$ that double vanish at $s$ form a vector space of codimension $n-1$ in $I_{1,d}(S)$ we work with configurations where the forms in $I_{1,d}(S)$ that double vanish at $s$ form a vector space of codimension $|S|-n+1$ in $P_{n,d}$. In other words we also require that the conditions of vanishing at all the points of $S$ are linearly independent. This is indeed our definition $d$-independence from before:

\begin{defin}\label{defindindep}
We call a finite set $S$ in $\rn$ \text{d-independent} if $S$ satisfies two following properties:\\
The set $S$ forces no additional zeroes for forms of degree $d$ that vanish on $S$,\\
For any $s \in S$ the forms that vanish to order 2 on $s$ and vanish on the rest of $S$
form a vector space of codimension $|S|+n-1$ in $P_{n,d}$.
\end{defin}

From Theorem \ref{gapzz} we obtain the following immediate corollary:

\begin{cor}\label{dindep}
Suppose that a set $S$ in $\rn$ is $d$-independent. Then the face $\p(S)$ is full dimensional in $I_{2,2d}(S)$.
\end{cor}

Using standard methods (vanishing determinants) it is easy to write the set of all configuration of $k$ points in $\mathbb{RP}^{n-1}$ that are $d$-independent as a complement of a closed algebraic set. Therefore we obtain the following proposition:

\begin{prop}\label{open}
The set of configuration of $k$ points in $\mathbb{RP}^{n-1}$ that are $d$-independent is open.
\end{prop}

In the following section we will actually construct an example of a $d$-independent set of cardinality $\binom{n+d-1}{d}-n$. This will show that $d$-independence is a condition that holds on an open set for all $k \leq \binom{n+d-1}{d}-n$, i.e. it is generic. However, we strongly suspect that $d$-independence should be generic in a stronger sense: it should hold on an open set whose closure is all of $\mathbb{RP}^{n-1}$, or in other words "almost any" configuration of $k\leq \binom{n+d-1}{d}-n$ points is $d$-independent.

Indeed we show in Section \ref{sixpoints} that any set of 6 points in $\mathbb{RP}^3$  in general linear position is 2-independent and in Lemma 2.6 of \cite{Rez2} it was shown that any set of 7 points in $\mathbb{RP}^2$ with no 4 on a line and not all on a quadratic is $3$-independent.

\section{A $d$-independent Set.}\label{dset}

Define $\bar{S}_{n,d}$ to be the set of points in $\rn $ that correspond to nonnegative integer partitions of $d$:

$$\bar{S_{n,d}}=\left\{(\alpha_1,\ldots,\alpha_n) \in \rn\hspace{.4mm} | \hspace{.9mm}\alpha_i \in \mathbb{Z}, \alpha_i \geq 0, \sum_{i=1}^n \alpha_i=d\right\}.$$

We can think of the points in $\bar{S}_{n,d}$ as all the possible exponent choices for monomials in $n$ variables of degree $d$. Therefore $\bar{S}$ contains $\binom{n+d-1}{d}$ points.

Also let $S_{n,d}$ be the set of points in $\rn$ that correspond to partitions of $d$ with at least 2 nonzero parts. The points in $S_{n,d}$ again correspond to monomials of degree $d$ but we do not allow monomials of the form $x_i^d$. Therefore $S_{n,d}$ contains $\binom{n+d-1}{d}-n$ points.

The following proposition is taken from \cite{Rez3} p.31 and has been known for at least a hundred years. We reproduce the proof below.

\begin{prop}\label{wtf}
There are no nontrivial forms in $n$ variables of degree $d$ that vanish on $\bar{S}_{n,d}$. In other words $I_{1,d}(\bar{S}_{n,d})=0$.
\end{prop}

\begin{proof}
For every point $s=(s_1,\ldots,s_n) \in \bar{S}_{n,d}$ we will construct a form $p_s \in P_{n,d}$ that vanishes at all points in $\bar{S}_{n,d}$ except for $s$. This shows that the conditions of vanishing at a point in $\bar{S}_{n,d}$ are linearly independent and since $|\bar{S}_{n,d}|=\dim P_{n,d}$ we see that $\dim I_{1,d}(\bar{S}_{n,d})=0$.

Let $M=x_1+\ldots+x_n$. For $i=1 \ldots n$ let $h_i$ be a form defined as follows:
$$h_i(x)=\prod_{k=0}^{s_i-1}(dx_i-kM).$$
\noindent It is clear that the degree of $h_i$ is $s_i$ and $h_i$ vanishes on all partitions in $\bar{S}_{n,d}$ with $i$-th part less than $s_i$. Now let $p_s$ be defined as:

$$p_s=\prod_{i=1}^{n}h_i.$$

\noindent The form $p_s$ has degree $\sum s_i=d$ and it does not vanish on $s$. However for any other partition of $d$ there will be a part that for some $i$ is less than $s_i$. Then $h_i$ will vanish for that $i$ and thus $p_s$ will vanish on any partition of $d$ except for $s$.
\end{proof}

Let $M=x_1+\ldots+x_n$. For $i=1, \ldots ,n$ define a form $Q_i$ as follows:

\begin{equation}\label{qi}
Q_i=\prod_{k=0}^{d-1} (dx_i-{kM}).
\end{equation}

 We observe that each $Q_i$ vanishes on $S_{n,d}$. Let $s=(s_1, \ldots, s_n)\in S_{n,d}$ and consider $Q_i(s)$. We know that $M(s)=d$ because points in $S_{n,d}$ are partitions of $d$, and therefore the term in the defining product of $Q_i$ that corresponds to $k=s_i$ will vanish at $s$, making $Q_i(s)=0$. Therefore $Q_i \in I_{1,d}(S_{n,d})$ for all $i=1 \ldots n$.

We will now show that $Q_i$ actually form a basis of $I_{1,d}(S)$. The fact that we have such a nicely factoring basis is what allows us to prove that $S_{n,d}$ is $d$-independent.

\begin{prop}
The forms $Q_i$ form a basis of $I_{1,d}(S_{n,d})$.
\end{prop}

\begin{proof}
We first show that $Q_i$ are linearly independent. Let $e_1 \ldots e_n$ be the standard basis vectors of $\rn$. Its easy to see that $Q_i(e_j)=0$ when $i \neq j$, since $x_i$ divides $Q_i$. On the other hand, $Q_i(e_i)=d!$. Therefore if a linear combination $\alpha_1Q_1+\ldots+\alpha_nQ_n=0$ for some $\alpha_i \in \mathbb{R}$ then by considering the value of this combination at $e_i$ we see that $\alpha_i=0$ and this works for all $i$. Thus $Q_i$ are linearly independent.

We now show that $Q_i$ span $I_{1,d}(S_{n,d})$. Let $p$ be a form in $I_{1,d}(S_{n,d})$ and let $\beta_i=p(e_i)$. Consider the form $$\bar{p}=p-\sum_{i=1}^{n}\frac{\beta_i}{d!}Q_i.$$ It is clear from the above that $\bar{p}$ vanishes on the standard basis vectors $e_i$. Therefore, $\bar{p}$ vanishes not only on $S_{n,d}$ but also on $\bar{S}_{n,d}$. By Proposition \ref{wtf} it follows that $\bar{p}=0$ and therefore $p$ is in the span of $Q_i$.

\end{proof}

We now show that the set $S_{n,d}$ satisfies the two conditions of $d$-independence from Definition \ref{defindindep}.

\begin{lemma}
The set $S$ forces no additional zeroes for forms of degree $d$.
\end{lemma}

\begin{proof}
Since we know that $Q_i$ form a basis of $I_{1,d}(S_{n,d})$ the statement of the lemma is equivalent to showing that $S_{n,d}$ is projectively equal to $\cap_{i=1}^n Z(Q_i),$
where $Z(Q_i)$ denotes the zero set of $Q_i$.

Let $v=(v_1, \ldots, v_n) \in \cap_{i=1}^n Z(Q_i)$ be a nonzero point and first suppose that $v_1+\ldots+v_n=0$. Then $M(v)=0$ and therefore by equation \eqref{qi} we see that $Q_i(v)=d^dv_i^d$. Since $Q_i(v)=0$ for all $i$ we see that $v=0$ which is a contradiction.

Now suppose that $v_1+\ldots+v_n \neq 0$. By homogeneity we can assume that $v_1+\ldots+v_n=d$. In this case, from equation \eqref{qi} it follows that $Q_i(v)=d^dv_i(v_i-1)\ldots(v_i-d+1).$ Since $Q(v_i)=0$ for all $i$, we see that each $v_i$ is a nonnegative integer between 0 and $d-1$ and $v_1+\ldots+v_n=d$. In other words $v \in S_{n,d}$.
\end{proof}

We know that $|S_{n,d}|=\binom{n+d-1}{d}-n$. For the second condition of $d$-independence we need to show that for any $s \in S_{n,d}$ the vector space of forms double vanishing on $s$ and vanishing on the rest of $S_{n,d}$ has codimension $|S_{n,d}|+n-1=\binom{n+d-1}{d}-1$ in $P_{n,d}$. Since $\dim P_{n,d}=\binom{n+d-1}{d}$ we need to show that the vector space of forms double vanishing at any $s\in S_{n,d}$ and vanishing on the rest of $S_{n,d}$ is one dimensional.

\begin{lemma}
For every point $s \in S_{n,d}$ there is a unique (up to a constant multiple) form in $I_{1,d}(S)$ singular at $s$.
\end{lemma}

\begin{proof}
Let $s=(s_1,\ldots, s_n) \in S_{n,d}$ and let $p \in I_{1,d}(S_{n,d})$ be a form singular at $s$.

Since $Q_i$ form a basis of $I_{1,d}(S_{n,d})$, we may assume that $p=\alpha_1Q_1+\ldots+\alpha_nQ_n$. Now let $A=(a_{ij})$ be a $n \times n$ matrix with entries $$a_{ij}=\frac{\partial Q_i}{\partial x_j}(s).$$

The statement of the lemma is equivalent to showing that the rank of $A$ is $n-1$. Recall from equation \eqref{qi} the definition of $Q_i$:

$$Q_i=\prod_{k=0}^{d-1} dx_i-{kM}.$$

The form $Q_i$ vanishes at $s$ because the term $dx_i-{s_iM}$ corresponding to $k=s_i$ vanishes at $s$. Therefore, the only nonzero term in $\displaystyle \frac{\partial Q_i}{\partial x_j}$ evaluated at $s$ will come from differentiating out $dx_i-{s_iM}$. Now let

$$P_{s_i}=\frac{Q_i}{dx_i-{s_iM}}.$$
We observe that $P_{s_i}(s) \neq 0$ since we removed from $Q_i$ the only factor that vanishes at $s$.

Recall that $M=x_1+\ldots+x_n$ and therefore if we differentiate out $dx_i-{s_iM}$ from $Q_i$ with respect to $x_j$ and evaluate it at $s$ we see that

 \begin{displaymath}
   \frac{\partial Q_i}{\partial x_j}(s) = \left\{
     \begin{array}{lr}
       P_{s_i}(s)\left(d-{s_j}\right) &  \text{if} \hspace{2mm}i=j\\
       -P_{s_i}(s)s_j  & \text{if} \hspace{2mm}i \neq j.
     \end{array}
   \right.
\end{displaymath}

Since $P_{s_i}(s) \neq 0$ we can divide the $i$-th row of $A$ by ${P_{s_i}(s)}$ to obtain matrix $B=(b_{ij})$ where

 \begin{displaymath}
   b_{ij} = \left\{
     \begin{array}{lr}
       d-s_j &  \text{if} \hspace{2mm}i=j\\
       -s_j  & \text{if} \hspace{2mm}i \neq j.
     \end{array}
   \right.
\end{displaymath}

We obtained $B$ from multiplying rows of $A$ by nonzero numbers and therefore the rank of $B$ is equal to the rank of $A$.

Since $s$ is a partition of $d$ it is clear that the vector consisting of all 1's is in the kernel of $B$. Now let $C=(c_{ij})$ be the matrix with $j$-th column having the same entry $s_j$, i.e. $c_{ij}=s_j$. We observe that the rank of $C$ is 1 and $B=dI-C$ where $I$ is the identity matrix. Therefore we know that $\text{rank} \hspace{.5mm} B \geq \text{rank} \hspace{.5mm} I - \text{rank} \hspace{.5mm} C=n-1$. Since we already found a vector in the kernel of $B$ it follows that the rank of $B$ is $n-1$.

\end{proof}

We have now shown that the set $S_{n,d}$ is $d$-independent and together with Proposition \ref{open} this shows that $d$-independence is a generic condition for sets $k$ points in $\rn$ with $k\leq \binom{n+d-1}{d}-n$. We now use this to find large gaps the faces of $\p$ and $\sq$.

\section{Large Dimensional Gaps between $\p(S)$ and $\sq(S)$}\label{gapsection}
We now establish large gaps between $\p(S)$ and $\sq(S)$.

\begin{theorem}
Let $S$ be a $d$-independent set of $k$ points in $\rn$. Then the dimension of $\p(S)$ is at least $\binom{n+2d-1}{2d}-kn$ while the dimension of $\sq(S)$ is at most $\binomi{\binom{n+d-1}{d}-k+1}{2}$.
\end{theorem}

\begin{proof}
The dimension of $I_{2,2d}(S)$ is at least $\binom{n+2d-1}{2d}-kn$ since we are imposing at most $kn$ linearly independent conditions by forcing forms to double vanish at all points of $S$. From Corollary \ref{dindep} we know that $\p(S)$ is full dimensional in $I_{2,2d}(S)$ and thus the bounds for the dimension of $\p(S)$ follows.

Since $S$ is $d$-independent we know that the dimension of $I_{1,d}(S)$ is $\binom{n+d-1}{d}-k$. We can have at most $\binomi{\dim I_{1,d}(S)+1}{2}$ linearly independent  products coming from $I_{1,d}(S)$ and therefore the dimension of $I^{[2]}_{2d}(S)$ is at most $\binomi{\binom{n+d-1}{d}-k+1}{2}$. Since $\sq(S)$ is contained in $I^{[2]}_{2d}(S)$ the bound for $\sq(S)$ follows.
\end{proof}

Let $G_{n,2d}(k)$ be the size of the gap that we can show exists between $\p(S)$ and $\sq(S)$ for a $d$-independent set $S$ of size $k$:

\begin{equation}
G_{n,2d}(k)=\binom{n+2d-1}{2d}-kn-\binomi{\binom{n+d-1}{d}-k+1}{2}.
\end{equation}

From Section \ref{dset} we know that there exist $d$-independent sets of any cardinality $k \leq \binom{n+d-1}{d}-n$. We want to know the smallest $k$ where we can show that the gaps exists and we would also like to know the largest gap size. In other words we want to find the least $k$ for which $G_{n,2d}(k) > 0$ and we want to find the maximum of $G_{n,2d}(k)$.

\begin{prop}\label{uglybounds}
The function $G_{n,2d}(k)$ is maximized at $k=\binom{n+d-1}{n}-n$. Its value and the largest gap are
\begin{equation}\label{lessugly}
\binom{n+2d-1}{2d}-n\binom{n+d-1}{d}+\binom{n}{2}.
\end{equation}
The smallest value of $k$ to make $G_{n,2d}(k)$ positive is the smallest integer strictly greater than:
\begin{equation}\label{fugliness}
\binom{n+d-1}{d}-n+\frac{1}{2}-\sqrt{\left(n-\frac{1}{2}\right)^2+2\binom{n+2d-1}{2d}-2n\binom{n+d-1}{d}}.
\end{equation}

\end{prop}

Before we prove Proposition \ref{uglybounds} we make several remarks. First we observe that the largest gap number
$$\binom{n+2d-1}{2d}-n\binom{n+d-1}{d}+\binom{n}{2}$$
is zero in all the cases where the cones $\p$ and $\sq$ are equal. However it is strictly positive in the cases where exist nonnegative forms that are not sums of squares. In the smallest cases $n=4$, $2d=4$ and $n=3$, $2d=6$ where $\p$ is strictly larger than $\sq$ the gap number is 1.

However, as either $n$ or $d$ we can see that the dimensional gap between faces of $\p$ and $\sq$ grows and asymptotically it approaches the full dimension of the vector space $P_{n,2d}$.

We note that the bound from Equation \ref{fugliness} simplifies remarkably for $n=3$. In this case we get the bound of $\binom{d+2}{2}-d-1$ and we need to take the smallest integer above that which leads to
$$k=\binom{d+2}{2}-d.$$
This is actually the correct bound for the case of $n=3$ and we hope to discuss this elsewhere.

However for $n \geq 4$ the formula does not appear to simplify and the bound given is not going to be optimal. This is due to the overcounting in the bound we use that for the dimension of the vector space of squares $I^{[2]}_{2d}(S)$. It would be interesting to improve the bound which should lead to the optimal value of $k$.

We note that for $k=\binom{n+d-1}{d}-n$, which leads to the largest gap, the bound on the dimension of $I^{[2]}_{2d}(S)$ is also optimal generically. We can see that from the example of $d$-independent set $S_{n,d}$ from Section \ref{dset}, which has this cardinality. Indeed in the case of $S_{n,d}$ it is not hard to show that all pairwise products of $Q_i$, which form the basis of $I_{1,d}(S_{n,d})$ are linearly independent in $P_{n,2d}$. This shows that the dimension of $I_{1,d}(S_{n,d})$ is $\binom{n+1}{2}$, which is exactly equal to the bound we use.

We now prove Proposition \ref{uglybounds}.
\begin{proof}
We observe that $G_{n,d}(k)$ is a quadratic function of $k$ with a negative leading coefficient. It is easy to show that the maximum of $G_{n,d}(k)$ occurs at $\displaystyle k=\binom{n+d-1}{d}-n+\frac{1}{2}$. Therefore the maximum value of $G_{n,2d}(k)$ for an integer $k$ will occur with $k=\binom{n+d-1}{d}-n$ and it is a matter of easy simplification to obtain equation \eqref{lessugly}.

The bound in equation \eqref{fugliness} comes from simply calculating the smallest root of $G_{n,2d}(k)$. We skip the routine application of the quadratic formula.
\end{proof}

We now fully describe the situation with respect to $\p(S)$ and $\sq(S)$ for 6 points sets $S$ in $\mathbb{R}^4$.

\section{Six points in $\mathbb{R}^4$} \label{sixpoints}

Let $S=\{s_1, \ldots, s_6\}$ be a set of six points in $\mathbb{R}^4$ in general linear position. We will show that $S$ is 2-independent. In particular this implies that the conditions of vanishing at $s_i \in S$ are linearly independent and therefore $\dim I_{1,2}(S)=\binom{5}{2}-6=4$. It follows that the dimension of the vector space $I^{[2]}_4(S)$ spanned by squares from $I_{1,d}(S)$ is at most $\binom{5}{2}=10$. We will show that the dimension of $I^{[2]}_4(S)$ is indeed 10.

On the other hand Alexander-Hirschowitz Theorem tells us that the dimension of $I_{2,4}(S)$ is $\binom{7}{4}-6 \cdot 4=11$ generically. It is not hard to show that for 6 points in $\mathbb{R}^4$ in general linear position this dimension count is actually correct. Therefore we should have a gap of 1 dimension between $Pos_{4,4}(S)$ and $Sq_{4,4}(S)$. In particular, there is a linear constraint that is satisfied by squares and not satisfied by double vanishing polynomials on $S$. There is also a fourth degree form that double vanishes on $S$ but is not in the span of squares. We identify the extra constraint and a double vanishing form not in the span of squares below.

To every $3$ element subset $T=\{t_1,t_2,t_3\}$ of $\{1, \ldots, 6\}$ we can associate the hyperplane $L_T$ spanned by the vectors $s_{t_1}$, $s_{t_2}$ and $s_{t_3}$.

We want to select a double covering of $s_1, \ldots, s_6$ by 4 hyperplanes of the form $L_T$ with some nice combinatorial properties. We select 4 triples $T_i$ such that any two of them intersect at exactly one element of $\{1,\ldots,6\}$ and each element is contained in precisely two triples. Here is an example of such a covering, which is not unique: $$T_1=\{123\}, \hspace{4mm}T_2=\{145\},\hspace{4mm} T_3=\{246\}, \hspace{4mm}T_4=\{356\}.$$ To every such covering we can associate the complementary covering, where we replace the triple $T_i$ with its complement $\overline{T}_i$. So, in our case, $\overline{T}_1=\{456\}$, $\overline{T}_2=\{236\}$, $\overline{T}_3=\{135\}$ and $\overline{T}_4=\{124\}$. We observe that the complementary covering also shares the property any two triples intersect in exactly one point and every point is contained in exactly two triples.

To each triple $T$ we associate the linear functional with kernel $L_T$. We can think of this functional as the inner product with the unit normal vector to $L_T$, which is unique up to a sign. The choice of sign will not make a difference to us. We let $u_i$ be a unit normal to $L_{T_i}$ and $v_i$ be a unit normal to $L_{\overline{T}_i}$:
\begin{align*}
u_i \hspace{2mm} \text{is a unit vector perpendicular to} \hspace{2mm} L_{T_i},\\
v_i \hspace{2mm} \text{is a unit vector perpendicular to} \hspace{2mm} L_{\overline{T}_i}.
\end{align*}

Vectors $u_i$ and $v_i$ form a pair of bases of $\mathbb{R}^4$. The key is to work with the dual configurations. We define $u_i^*$ to be vectors such that
 \begin{displaymath}
   \ip{u_i^*}{u_j} = \left\{
     \begin{array}{lr}
       1 &  \text{if} \hspace{2mm}i=j\\
       0  & \text{if} \hspace{2mm}i \neq j.
     \end{array}
   \right.
\end{displaymath}
One way to think about $u_i^*$ is that if we form matrix $U$ with rows $u_i$ then $u_i^*$ form the columns of $U^{-1}$. We define vectors $v_i^*$ in the same way for $v_i$.

We will show that the four forms
\begin{align*}
Q_1(x)=\ip{x}{u_1}\ip{x}{v_1}, \hspace{4mm} Q_2(x)=\ip{x}{u_2}\ip{x}{v_2}\\
Q_3(x)=\ip{x}{u_3}\ip{x}{v_3}, \hspace{4mm} Q_4(x)=\ip{x}{u_4}\ip{x}{v_4}
\end{align*}
form a basis of $I_{1,2}(S)$. This factoring basis will allow us to prove 2-independence of $S$, and pairwise products $Q_iQ_j$ with $i \leq j$ will form a basis of $I^{[2]}_4(S)$.

Then we will show that the fourth degree form
$$R=\ip{x}{u_1}\ip{x}{u_2}\ip{x}{u_3}\ip{x}{u_4}$$
is singular at each of $s_i$ but it is not in $I^{[2]}_4(S)$.

Finally, let $Q$ be a form in $I^{[2]}_4(S)$. For any $i$ the form $Q$ satisfies:
\begin{equation}\label{equationextraconstraint}
\ip{v_i^*}{u_i}^2Q(u_i^*)=\ip{u_i^*}{v_i}^2Q(v_i^*).
\end{equation}
On the other hand $R$ will not satisfy any of these constraints, which shows that any one of these constraints is independent of being singular at points $s_i$. Of course, by the dimension count there is only 1 "true" extra linear constraint so we can take any one and the rest will cease being "new".

Before we continue with the proofs we will give an explicit example of the extra form and the explicit linear constraint.

\subsection{Explicit Example.}\label{sectionexplicitexample} Let $s_1=(0,0,1,1)$, $s_2=(0,1,0,1)$, $s_3=(0,1,1,0)$, $s_4=(1,0,0,1)$, $s_5=(1,0,1,0)$ and $s_6=(1,1,0,0)$. This is the set from Example \ref{sectionfirstgap}. Our particular numbering of points is chosen to mesh well with our system of covering triples $T_i$ and $\overline{T}_i$.

Vector $u_1$ comes from triple $123$ and therefore is a unit vector perpendicular to $s_1$, $s_2$, and $s_3$ and we may chose choose $u_1=e_1$, the first standard basis vector. Similarly $u_2$ comes from $145$ and is normal to $s_1$, $s_4$ and $s_5$. We may choose $u_2=e_2$. In the same way $u_3=e_3$ and $u_4=e_4$.

For the vectors $v_i$, $v_1$ comes from $456$ and we may choose $v_1=\frac{1}{2}(1,-1,-1,-1)$. In the same way $v_2=\frac{1}{2}(-1,1,-1,-1)$, $v_3=\frac{1}{2}(-1,-1,1,-1)$ and $v_4=\frac{1}{2}(-1,-1,-1,1)$

The extra form $R=\ip{x}{u_1}\ip{x}{u_2}\ip{x}{u_3}\ip{x}{u_4}$ becomes $$R=\ip{x}{e_1}\ip{x}{e_2}\ip{x}{e_3}\ip{x}{e_4}=x_1x_2x_3x_4,$$
as promised in the Example $\ref{sectionfirstgap}$.

We note that the sets $\{u_i\}$ and $\{v_i\}$ form 2 orthogonal bases in $\mathbb{R}^4$ and therefore are self-dual. It follows that $u_i^*=u_i$ and $v_i^*=v_i$. The extra constraint from Equation \eqref{equationextraconstraint} becomes: $\frac{1}{4}Q(u_1)=\frac{1}{4}Q(v_1)$ or after rewriting and using homogeneity of $Q$:
$$16Q(1,0,0,0)=Q(1,-1,-1,-1),$$
again as claimed in Example \ref{sectionfirstgap}.

\subsection{Proofs.} The vectors $u_i$ and $v_i$ are not just two arbitrary sets of bases of $\mathbb{R}^4$. Since they come from a configuration of 6 points in general position they have some structure. The following simple lemma will be crucial to our proofs.

\begin{lemma}\label{inprod}
For all i,j the following hold:
\begin{equation*}
\ip{u_i}{v_j^*}\neq 0 \hspace{2mm} \text{and} \hspace{2mm} \ip{v_i}{u_j^*}\neq 0.
\end{equation*}
\end{lemma}
\begin{proof}
By symmetry it will suffice to prove only one of the two assertions. Also, by symmetry it will suffice to show that $\ip{u_1^*}{v_1}\neq 0$ and $\ip{u_1^*}{v_2}\neq 0$.

Let's suppose that $\ip{u_1^*}{v_1}=0$ then it follows that $v_1$ is in the span of $u_2,u_3,u_4$. Let $$v_1=\alpha_2u_2+\alpha_3u_3+\alpha_4u_4.$$ Now lets consider the inner product $\ip{v_1}{s_4}$. Recall that $v_1$ came from the triple 456, $u_2$ from 145, $u_3$ from 246 and $u_4$ from 356. It follows that $$\ip{v_1}{s_4}=0=\alpha_4\ip{s_4}{u_4}.$$

Since the points $s_i$ are in general position it follows that $\ip{s_4}{u_4}\neq 0$ and therefore $\alpha_4=0$. By considering inner products of $v_1$ with $s_5$ and $s_6$ we can also show that $\alpha_2=\alpha_3=0$ which yields a contradiction.

Similarly, if $\ip{u_1^*}{v_2}=0$ then $v_2$ is in the span of $u_2,u_3,u_4$. Let $$v_2=\alpha_2u_2+\alpha_3u_3+\alpha_4u_4.$$ Recall that $v_2$ came from the triple 236, $u_2$ from 145, $u_3$ from 246 and $u_4$ from 356. By the same argument we can establish that $\alpha_2=0$ by using inner products with $s_6$. Then we use inner product with $s_2$ to show that $\alpha_4=0$ and we will arrive at a contradiction.

\end{proof}

\begin{lemma}
Let $Q_i(x)=\ip{x}{u_i}\ip{x}{v_i}$ for $i=1 \dots 4$. The forms $Q_i$ form a basis of $I_{1,2}(S)$. Furthermore the pairwise products $Q_iQ_j$ with $i \leq j$ form a basis of $I^{[2]}_4(S)$ and the dimension of $I^{[2]}_4(S)$ is 10.
\end{lemma}
\begin{proof}
It is not hard to show that $I_{1,2}(S)$ has dimension 4. Therefore it suffices to show that the polynomials $Q_i$ are linearly independent. Consider the values of $Q_i$ at the points $u_i^*$.

From the definition of the dual points $u_i^*$ and Lemma \ref{inprod} it follows that $Q_i(u_i^*)=\ip{u_i^*}{v_i} \neq 0$ and $Q_i(u_j^*)= 0$ when $i \neq j$. Therefore, if $P=\alpha_1Q_1+\alpha_2Q_2+\alpha_3Q_3+\alpha_4Q_4=0$ then by considering $P(u_i^*)$ we can see that $\alpha_i$ is 0 for each $i$ and therefore $Q_i$ are linearly independent.

Now lets consider pairwise products $Q_iQ_j$ for $i \leq j$. These forms are clearly in $I^{[2]}_4(S)$, and we need to show their linear independence. Of all the pairwise products only $Q^2_i$ does not vanish on $u^*_i$. Therefore the squares $Q_i^2$ are linearly independent from all other pairwise products and we only need to show linear independence of $Q_iQ_j$ for $i < j$.

By Lemma \ref{inprod} only the products $Q_iQ_j$ vanish on $u^*_i$ to order 1. If both indices are distinct from $i$ then the product vanishes to order 2. Therefore if forms $Q_iQ_j$ are linearly dependent it follows that the forms $Q_iQ_j$ for some fixed $i$ are linearly dependent. We can factor out $Q_i$ and it follows that the forms $Q_j$ are linearly dependent. This is a contradiction.

Since pairwise products $Q_iQ_j$ span $I^{[2]}_4(S)$ and are linearly independent it follows $\dim I^{[2]}_4(S)=10$.

\end{proof}

We are now ready to show 2-independence of $S$.

\begin{prop}
Let $S$ be a set of 6 points in $R^4$ in general linear position. Then $S$ is 2-independent.
\end{prop}

\begin{proof}
We first show that $S$ forces no additional zeroes on quadratic forms. Recall that $Q_i=\ip{x}{u_i}\ip{x}{v_i}$ and the forms $Q_i$ form a basis of $I_{1,2}(S)$. It will suffice to show that the forms $Q_i$ have no common zeroes outside of $S$.

Let $z$ be a nonzero point in the intersection $\cap_{i=1}^4 Z(Q_i).$ It follows that for each $i$ we either have $\ip{z}{u_i}=0$ or $\ip{z}{v_i}=0$. Since $u_i$ and $v_i$ form a basis of $\mathbb{R}^4$ the vector $z$ cannot be orthogonal to all four $u_i$ or $v_i$. If $\ip{z}{u_i}=0$ for three indices $i$, which we may assume without loss of generality to be 1,2 and 3, then it follows that $z$ is a multiple of $u_4^*$. But then $\ip{z}{u_4} \neq 0$  and from Lemma \ref{inprod} we know that $\ip{z}{v_4} \neq 0$. Therefore $Q_4(z) \neq 0$, which is a contradiction.

Therefore it must happen that $z$ is orthogonal to two $u_i$ and two $v_i$. Again without loss of generality we may assume that $z$ is orthogonal to $u_1$, $u_2$, $v_3$ and $v_4$. Since $u_1$ comes from the triple $123$, $u_2$ comes from $123$, $v_3$ comes from $135$ and $124$ it follows that $z$ is in the intersection of spans of $\{s_1, s_2, s_3\}$, $\{s_1, s_4, s_5\}$, $\{s_1,s_3,s_5\}$ and $\{s_1,s_2,s_4\}$. Since the points $s_i$ are in general linear position it follows that $s_1$ spans this intersection. The other points $s_i$ arise in the same manner from choosing different pairs of $u_i$'s and $v_i$'s.

For the second condition of $2$-independence we need to show that for any $s_i \in S$ there exists a unique (up to a constant multiple) form in $I_{1,2}(S)$ that is singular at $s_i$. Again by symmetry we only need to prove this for $s_1$. By construction $s_1$ is orthogonal to $u_1$, $u_2$, $v_3$ and $v_4$. Therefore it follows that
$\nabla Q_1(s_1)=\ip{v_1}{s_1}u_1$, $\nabla Q_2(s_1)=\ip{v_2}{s_1}u_2$, $\nabla Q_3(s)=\ip{u_3}{s_1}v_3$ and $\nabla Q_4(s)=\ip{u_4}{s_1}v_4$. The coefficients of vectors $u_1$,$u_2$, $v_3$ and $v_4$ are nonzero and since $s_i$ are in general linear position it follows that $u_1$,$u_2$, $v_3$ and $v_4$ span the vector space $s_1^{\perp}.$ Therefore there is only one (up to a constant multiple) linear combination of gradients of $Q_i$ that vanishes at $s_1$.
\end{proof}

Now we show that the fourth degree form
$$R=\ip{x}{u_1}\ip{x}{u_2}\ip{x}{u_3}\ip{x}{u_4}$$ is not in the span of squares from $I_{1,2}(S)$.

\begin{prop}
Let $R=\ip{x}{u_1}\ip{x}{u_2}\ip{x}{u_3}\ip{x}{u_4}$. The $R$ is not in $I^{[2]}_4(S)$.
\end{prop}

\begin{proof}
We know that products $Q_iQ_j$ with $i \leq j$ form a basis of $I^{[2]}_4(S)$. We observe that $R(u_i^*)=0$ for all $i$, and the only form from the spanning set that doesn't vanish at $u_i^*$ is $Q_i^2$. Therefore, if we assume that $R$ is spanned by $Q_iQ_j$ then $R$ is spanned by products $Q_iQ_j$ with $i$ not equal to $j$.

Now lets look at $R(v_k^*)$. By Lemma \ref{inprod} we know that $R(v_k^*) \neq 0$. However, $Q_iQ_j(v_k^*)=0$ since $\ip{v_i^*}{v_k}=0$ for $i \neq k$. Therefore we arrive at a contradiction.
\end{proof}

Now we focus on an explicit linear constraint on forms in $I^{[2]}_4(S)$ that is independent of vanishing gradients on the points $s_i$.
To establish the constraint it is enough to look at forms in $I^{[2]}_4(S)$ that are squares, since they span $I^{[2]}_4(S)$.

\begin{prop}
Let $Q=\alpha_1Q_1+\alpha_2Q_2+\alpha_3Q_3+\alpha_4Q_4$. Then $Q$ satisfies
$$\ip{v_i^*}{u_i}Q(u_i^*)=\ip{u_i^*}{v_i}Q(v_i^*)$$
for all $i$. It follows that $Q^2$ satisfies:
$$\ip{v_i^*}{u_i}^2Q^2(u_i^*)=\ip{u_i^*}{v_i}^2Q^2(v_i^*)$$
On the other hand the form $R$ will not satisfy any of these constraints.
\end{prop}

\begin{proof} For any $i$,
$$Q(u_i^*)=\alpha_iQ_i(u_i^*)=\alpha_i\ip{u_i^*}{v_i},$$ and $$Q(v_i^*)=\alpha_iQ_i(v_i^*)=\alpha_i\ip{v_i^*}{u_i}.$$ Therefore $$\ip{v_i^*}{u_i}Q(u_i^*)=\ip{u_i^*}{v_i}Q(v_i^*).$$
The constraint for $Q^2$ now follows.

On the other hand by definition of $u_i^*$ and Lemma \ref{inprod} we know that $$R(u^*_i)=0 \hspace{.7cm} \text{while} \hspace{.7cm} R(v_i^*)\neq 0 \hspace{.7cm} \text{and} \hspace{.7cm}\ip{u_i^*}{v_i}\neq 0.$$ Therefore $R$ does not satisfy the relation for any $i$ and any one of these relations is independent of gradient vanishing at points $s_i$.
\end{proof}

\section{Inequalities in the General Setting}\label{sectiongenineq}

Let $S$ be a finite set of points in $\rn$ and suppose that the face $\p(S)$ of the cone of nonnegative forms has a higher dimension than the face $\sq(S)$ of the cone of sums of squares. Then it follows that there must be extra linear constraints that are satisfied by sums of squares in $\sq(S)$, but are not satisfied by the nonnegative forms in $\p(S)$. Since $\sq(S)$ spans the vector space $I^{[2]}_{2d}(S)$ the extra constraints hold for any polynomial that is spanned by squares.

Let $l_1,\ldots,l_k$ be linear functionals on $P_{n,2d}$ that form a basis of the set of extra constraints. For any $p \in \sq$ we know that

\begin{equation}\label{forcedcontraints}
\text{if} \hspace{2mm} p(s)=0 \hspace{2mm} \text{for all} \hspace{2mm} s \in S\hspace{2mm} \text{then} \hspace{2mm} l_1(p)=\ldots=l_k(p)=0.
\end{equation}
Let $R_S$ be a quadratic functional on $P_{n,2d}$ given by the sum of squares of linear functionals $l_i$:
$$R_S(p)=\sum_{i=1}^k l_i^2(p).$$

\noindent Let $M_S$ be the linear functional on $P_{n,2d}$ given by summing the values of a form $p$ on the points $s \in S$:

$$M_S(p)=\sum_{s \in S} p(s).$$

\noindent Finally let $T(p)$ be the linear functional given by averaging a form $p$ over the unit sphere:

$$T(p)=\int_{\sph}p \hspace{.7mm} d\sigma,$$

\noindent where $\sigma$ is the uniform probability measure on $\sph$.

We claim that there exists $\alpha >0$ such that for all $p \in \sq$
\begin{equation} \label{genineq}
\alpha M_S(p)T(p)-R_S(p) \geq 0.
\end{equation}

We briefly explain why there exists $\alpha$ that makes inequality \eqref{genineq} hold for all $p \in \sq(S)$. First we observe that we can restrict ourselves to the case of forms of average 1 on the unit sphere, i.e. the case of $T(p)=1$.

The linear functional $M_S$ is clearly nonnegative on $\sq$ 
. When $M_S(p)=0$ we know by \eqref{forcedcontraints} all the functionals $l_i(p)$ must vanish and therefore $R_S(p)=0$. Since $M_S(p)>0$ for all $p \in \sq$ that are not in $\sq(S)$ it follows that the only obstacle to finding an appropriate $\alpha$ has to come from looking infinitesimally close to the face $\sq(S)$ where $M_S(p)=0$. Owing to the quadratic nature of the cone of sums of squares $\sq$ we will be able to argue the existence of $\alpha$ and we will provide an explicit example of such inequality in Section \ref{sectionexplicitineq}.

We observe that regardless of $\alpha$ the inequality \eqref{genineq} will not hold for some $f \in \p$. Let $f$ be a form such that $f$ is in $\p(S)$ but $f$ is not in the vector space $I^{[2]}_{2d}(S)$ spanned by $\sq(S)$. We have a dimensional gap between $\p(S)$ and $\sq(S)$  and thus we know that such $f$ exist. Since $f \in \p(S)$ it vanishes at every point of $S$ and therefore $M_S(p)=0$. Also, since $f$ is not in $I^{[2]}_{2d}(S)$ it follows that at least one of the functionals $l_i$ is not zero on $f$ and therefore $R_S(f) >0$. We thus see that $M_S(f)T(f)-R_{S}(f) < 0$.

We note that the linear functional $T(p)$ is used to homogenize the inequality and we could have used any other linear functional that is strictly positive on all nonzero forms in $\sq$ instead.

We first need a preliminary lemma that is very similar in flavor to the Extension Lemma \ref{ext}.

\begin{lemma}\label{prelim}
Let $Q_1$ and $Q_2$ be two quadratic forms on a real vector space $V$, such that $Q_1$ is positive semidefinite and $Q_2$ is positive definite on $V$. Let $R$ be a sum of squares of quadratic forms $p_i$ on $V$: $$R=\sum_i p_i^2,$$
and suppose further that $R$ vanishes whenever $Q_1$ vanishes: $Z(Q_1) \subseteq Z(R)$. Then there exists $\alpha >0$ such that the form
$Q_{\alpha}=\alpha Q_1 Q_2 -R$ is nonnegative on $V$:
\begin{equation}
\label{qu}
Q_{\alpha}(v)=\alpha Q_1(v) Q_2(v) -R(v) \geq 0
\end{equation}
for all $v \in V$.
\end{lemma}

\begin{proof}
The inequality \eqref{qu} is homogeneous and therefore it suffices to prove it for all $v$ on the unit sphere $S_V$ of $V$. Since $Q_1$ is a positive semidefinite quadratic form we know that $Q_1$ vanishes on a subspace $W$ of $V$.  We note that both $Q_1Q_2$ and $R$ vanish to order 2 on $W$, since zero is a global minimum for both forms.

Let's pick a point $w \in W$ that is also on the unit sphere. The Hessian $H_{Q_1Q_2}(w)$ of $Q_1Q_2$ at $w$ has $W$ as the kernel and is positive definite on $W^{\perp}$. This follows from the fact that $Q_1$ is positive semidefinite and $Q_2$ is positive definite. The same is true for the Hessian $H_R(w)$ of $R$ at $w$ since $R$ also vanishes on $W$ and zero is a global minimum of $R$. Therefore by compactness of the unit sphere we can find $\alpha_1$ such that $\alpha_1H_{Q_1Q_2}(w)-H_R(w)$ is positive definite on $W^{\perp}$ for any $w\in W$ on the unit sphere. It follows that the form $Q_{\alpha_1}$ is nonnegative on all points of the unit sphere that are distance at most $\delta$ from $W$, for some $\delta >0$.

Let's consider points $v \in S_v$ that are at least $\delta$ away from the subspace $W$ on which $Q_1$ vanishes. We know that for these points $v$ we will have $Q_1(v)Q_2(v) \geq \epsilon$ for some $\epsilon >0$. Then we can multiply $Q_1Q_2$ by a sufficiently large $\alpha_2$ so that $Q_{\alpha_2}$ will be positive on the points $v$ that are $\delta$ away from $W$. We choose $\alpha=\max (\alpha_1, \alpha_2)$ and $Q_{\alpha}$ is nonnegative on the whole unit sphere $S_V$.

\end{proof}

Now we prove the existence of $\alpha$ that makes inequality \eqref{genineq} true.
\begin{theorem}\label{theoremineqgen}
There exists $\alpha >0$ such that for all $p \in \sq$
\begin{equation*}
\alpha M_S(p)T(p)-R_S(p) \geq 0.
\end{equation*}
\end{theorem}

\begin{proof}
Let $\overline{Sq}_{n,2d}$ be the section of the cone $\sq$ with hyperplane of forms of integral 1 on the unit sphere $\sph$:
$$\overline{Sq}_{n,2d}=\left\{p \in \sq \st \int_{\sph}p \hspace{.7mm} d\sigma=1 \right\}$$.

We begin by observing that it suffices to prove \eqref{genineq} for $p \in \overline{Sq}_{n,2d}$, so we restrict our attention to the case $T(p)=1$. We want to find $\alpha >0$ such that for all $p \in \overline{Sq}_{n,2d}$
\begin{equation}
\alpha M_S(p) - R_S(p) \geq 0.
\end{equation}

We note that $R_S$ is a sum of squares of linear functionals and therefore it is a convex functional on $P_{n,2d}$. Then it follows that the functional $M_S-R_S$ is concave. Since $\overline{Sq}_{n,2d}$ is a compact convex set the functional $M_S-R_S$ attains its minimum at an extreme point of $\overline{Sq}_{n,2d}$. We know that an extreme point of $\overline{Sq}_{n,2d}$ must be a square. Therefore it suffices to show \eqref{genineq} for squares.

We observe that $M_S(g^2)=\sum_{s \in S} g^2(S)$ is a positive semidefinite quadratic form on $P_{n,d}$ while $T(g^2)=\int_{\sph}g^2 \hspace{.7mm} d\sigma$ is a positive definite quadratic form on $P_{n,d}$. Also, $R_S(g^2)=\sum_i l_i^2(g^2)$ is a sum of squares of quadratic forms and $R_S$ vanishes whenever $M_S$ vanishes. Therefore we can apply Lemma \ref{prelim} and the existence of $\alpha$ follows.
\end{proof}

We note that the problem of finding the value of $\alpha$ for a fixed set $S$ and degree $d$ is a semidefinite programming problem and can be solved fast numerically using semidefinite program solving packages.

\section{An Explicit Inequality}\label{sectionexplicitineq}

We now derive an explicit inequality along the lines described in the previous section. We take $S=(s_{ij})$ to be the set of 6 points in $\mathbb{R}^4$ with $\frac{1}{\sqrt{2}}$ in coordinates $i$ and $j$ and 0 in the other two coordinates. This is exactly the set from Example \ref{sectionfirstgap} but we now make the vectors have unit length.

Let $a=(1,0,0,0)$ and $b=(\frac{1}{2},-\frac{1}{2},-\frac{1}{2},-\frac{1}{2})$. We recall from Example \ref{sectionfirstgap} and Section \ref{sectionexplicitexample} that the forms $p \in I^{[2]}_{4}(S)$ satisfy the extra linear constraint $p(a)-p(b)=0$.

We now recall from the previous section definitions of functionals $M_S$, $T$ and $R_S$ for $f \in P_{4,2}$:
\begin{align*}
M_S(f)&=\sum_{s_{ij} \in S} f^2(s_{ij}),\\
T(f)&=\int_{S^3}f^2 \hspace{.2mm} d\sigma,\\
R_S(f)&=(f^2(a)-f^2(b))^2.
\end{align*}
The operators are acting on $f^2$, since we know by the proof of Theorem \ref{theoremineqgen} that we only need to establish the inequality for squares.

We will show that
$$15 M_S(f)T(f)-P_S(f) \geq 0$$
for all forms $f \in P_{4,2}$  and 15 is the smallest value of $\alpha$ that makes the inequality hold.

In order to do this we will explicitly calculate the functionals $M_S$, $T$ and $R_S$ using the following inner product on $P_{4,2}$:

$$\ip{f}{g}=\int_{S^3} fg \hspace{.7mm} d\sigma.$$

\noindent For every point $v \in \mathbb{R}^4$ there exists a unique polynomial $P_v \in P_{4,2}$ such that

$$\ip{f}{P_v}=f(v) \hspace{2mm} \text{for all} \hspace{2mm} f \in P_{4,2}.$$

\noindent It is not hard to show that

\begin{equation}\label{innerprod}
P_{v}(x)=12\ip{x}{v}^2-2\|v\|^2\|x\|^2.
\end{equation}

\noindent Also by definition of $P_v$,

\begin{equation}
\ip{P_v}{P_w}=P_v(w)=12\ip{v}{w}^2-2\|v\|^2\|w\|^2.
\end{equation}

\noindent For more information on the integral inner product and the polynomials $P_v$ see \cite{Muller}. For the points $s_{ij}$ we will denote $P_{s_{ij}}$ simply by $P_{ij}$.

Let's first analyze the quadratic form $$M_S(f)=\sum_{s_{ij} \in S} f^2(s_{ij}),$$ for $f \in P_{2,4}$. This is a positive semidefinite quadratic form of rank 6 on $P_{4,2}$. Using our polynomials $P_{ij}$ we can write is as $$M_S(f)=\sum \ip{f}{P_{ij}}^2.$$

We let $V$ be the span of $P_{ij}$ and then $V^{\perp}$ is the kernel of $M_S$. Using equation \eqref{innerprod} it is a matter of routine calculation to show the quadratic form $M_S$ has two eigenspaces $V_1$ and $V_2$. The eigenspace $V_1$ corresponds to eigenvalue 12 and can be chosen to have the following orthogonal basis:
\begin{align*}
&v_1=\sum P_{ij}, \hspace{4mm} v_2=P_{12}+P_{13}+P_{14}-P_{23}-P_{24}-P_{34},\\
&v_3=P_{12}+P_{24}-P_{13}-P_{34}, \hspace{4mm} v_4=P_{12}+P_{13}+2P_{23}-2P_{14}-P_{24}-P_{34}.\\
\end{align*}
The eigenspace $V_2$ corresponds to eigenvalue $6$ and has orthogonal basis:
\begin{equation*}
v_5=P_{12}+P_{34}-P_{23}-P_{14}, \hspace{4mm} v_6=\sum P_{ij}-3(P_{13}+P_{2,4}).
\end{equation*}

If we pick as an orthogonal basis unit vectors in the direction of eigenvectors $v_i$ then the form $M_S$ becomes:
$$M_S(f)=12(x_1^2+x_2^2+x_3^2+x_4^2)+6(x_5^2+x_6^2) \hspace{2mm} \text{where} \hspace{2mm} x_i=\ip{f}{\frac{v_i}{\|v_i\|}}.$$
Now we analyze $R_S(f)=(f^2(a)-f^2(b))^2.$ We can rewrite this as $$P_S(f)=\left(\ip{f}{P_a}^2-\ip{f}{P_b}^2\right)^2.$$
\noindent Then
$$R_S(f)=\left(\ip{f}{P_a}^2-\ip{f}{P_b}^2\right)^2=\ip{f}{P_a+P_b}^2\ip{f}{P_a-P_b}^2.$$
Therefore $R_S(f)$ is a product of two rank 1 quadratic forms.

We know from Section \ref{sixpoints} that for all $f$ with $M_S(f)=0$ we also have $$f(a)=\ip{P_a}{f}=\ip{P_b}{f}=f(p_b).$$ It follows that $P_a-P_b$ is in the span of $P_{ij}$. It is easy to check that $$P_{a}-P_{b}=\frac{v_2}{2} \hspace{2mm} \text{and} \hspace{2mm} \|P_a-P_b\|^2=18.$$
\indent To deal with $P_a+P_b$ we note that its projection onto the span of $P_{ij}$ is equal to $\frac{1}{6}\sum P_{ij}=\frac{v_1}{6}$. It is easy to check using \eqref{innerprod} that $$\|P_a+P_b\|^2=22 \hspace{4mm} \text{while} \hspace{4mm} \|\frac{v_1}{6}\|^2=12.$$
\noindent We can therefore write $\displaystyle P_a+P_b=\frac{v_1}{6}+v_7$ where $v_7$ is a vector is the kernel of $M_S(f)$ and $\|v_7\|^2=22-12=10$. If we extend our basis $\left\{v_i/\|v_i\|\right\}$ of the span of $P_{ij}$ by adding the vector $v_7/\|v_7\|$ and 3 more unit vector to make a basis of $P_{4,2}$ the we can write:

\begin{equation}
R_S(f)=18x_2^2(12x_1^2+10x_7^2) \hspace{2mm} \text{where} \hspace{2mm} x_i=\ip{f}{\frac{v_i}{\|v_i\|}}.
\end{equation}

Since we chose the integral inner product the form $T$ is simply: $$T=\sum_{i=1}^{10}x_i^2,$$
for any choice of orthogonal basis. Therefore $$M_S(f)T(f)=\left(12(x_1^2+x_2^2+x_3^2+x_4^2)+6(x_5^2+x_6^2)\right)\sum_{i=1}^{10}x_i^2.$$

If we want to choose $\alpha$ such that $\alpha M_S(f)T(f)-R_S^2(f)>0$ then we need to choose it so that the coefficient of $x_{2}^2x_7^2$ is nonnegative. If follows that we need $\alpha \geq 18 \cdot 10/12=15,$ and it is easy to see that $\alpha=15$ indeed will suffice.

\end{document}